\documentclass[12pt,reqno]{amsart}
\usepackage{amsmath,amssymb,amsfonts,amscd,latexsym,amsthm,mathrsfs,verbatim,comment}
\usepackage{enumerate}
\newtheorem{theorem}{Theorem}
\newtheorem{lemma}{Lemma}
\theoremstyle{remark}
\newtheorem*{remark}{\bf Remark}
\let\lf\lfloor
\let\rf\rfloor
\begin{document}

\title{Continued fractions of cubic irrationalities}

\author{Wadim Zudilin}
\address{Department of Mathematics, IMAPP, Radboud University, PO Box 9010, 6500~GL Nij\-me\-gen, Netherlands}
\urladdr{https://www.math.ru.nl/~wzudilin/}

\date{28 November 2023}

\subjclass[2020]{Primary 11A55; Secondary 11J68, 11J70, 11R16}

\begin{abstract}
We highlight some facts about continued fractions of real cubic irrationalities.
\end{abstract}

\maketitle

Given a real number $\alpha$, we follow the standard procedure to define its continued fraction.
Take $a_1=\lf\alpha\rf$ and, if $\alpha$ is not an integer, then write it in the form $\alpha=a_1+1/\alpha_2$, where $\alpha_2>1$ is again a real number.
Inductively, we choose $a_n=\lf\alpha_n\rf$ and $\alpha_n=a_n+1/\alpha_{n+1}$ with $\alpha_{n+1}>1$ if $\alpha_n$ is not an integer.
The procedure terminates at some step (that is, eventually we get an integer $\alpha_n=a_n>1$) if our starting number $\alpha$~is rational; otherwise the recursive algorithm produces an infinite continued fraction
\[
\alpha=[a_1,a_2,\dots,a_n,\dots]=\lim_{n\to\infty}[a_1,a_2,\dots,a_n]
\]
with an irrational $\alpha_n>1$ at each step.
Here and in what follows
\[
[a_1,a_2,\dots,a_n]
=a_1+\dfrac1{a_2+\dfrac1{a_3 +{\atop\ddots\,\displaystyle+\dfrac1{a_n}}}}
\]
is a finite continued fraction built on partial quotients $a_1,\dots,a_n$;
it can be given the form $[a_1,a_2,\dots,a_n]=p_n/q_n$, where
\begin{equation}
\begin{pmatrix} p_n & p_{n-1} \\ q_n & q_{n-1} \end{pmatrix}
=\gamma_n=\begin{pmatrix} a_1 & 1 \\ 1 & 0 \end{pmatrix}
\begin{pmatrix} a_2 & 1 \\ 1 & 0 \end{pmatrix}
\dotsb\begin{pmatrix} a_{n-1} & 1 \\ 1 & 0 \end{pmatrix}
\begin{pmatrix} a_n & 1 \\ 1 & 0 \end{pmatrix}
\label{eq-gamma}
\end{equation}
for $n\ge1$ and conventionally $p_0=1$, $q_0=0$ (see, for example, \cite[Chapter~4]{Zu23}).
The above procedure for an irrational $\alpha$ means that
\[
\alpha=[a_1,a_2,\dots,a_n,\alpha_{n+1}] \quad\text{for}\; n=0,1,2,\dots,
\]
with $a_n<\alpha_n<a_n+1$ for $n\ge1$ and $a_n\ge1$ for $n\ge2$;
the rational number $p_n/q_n=[a_1,a_2,\dots,a_n]$ is called the $n$th convergent of~$\alpha$, while $a_n$ is the $n$th partial quotient.
The related matrix $\gamma_n\in GL_2(\mathbb Z)$ in \eqref{eq-gamma} is called the $n$th continued transformation of the number~$\alpha$; then
\begin{equation}
\alpha=\gamma_n\alpha_{n+1}=\frac{p_n\alpha_{n+1}+p_{n-1}}{q_n\alpha_{n+1}+q_{n-1}},
\label{eq-lf1}
\end{equation}
where more generally for $\gamma=\big(\begin{smallmatrix} a & b \\ c & d \end{smallmatrix}\big)\in GL_2(\mathbb Z)$ and $t\in\mathbb R$ we denote by
\[
\gamma t=\frac{at+b}{ct+d}
\]
the $GL_2(\mathbb Z)$-action on $\mathbb R$ by the induced linear-fractional transformations.
In particular, all $\alpha=\alpha_1,\alpha_2,\dots,\alpha_n,\dots$ are equivalent (with respect to this action);
we can also invert the relation \eqref{eq-lf1}:
\begin{equation}
\alpha_{n+1}=\gamma_n^{-1}\alpha=-\frac{q_{n-1}\alpha-p_{n-1}}{q_n\alpha-p_n}.
\label{eq-lf2}
\end{equation}

If $\alpha\in\mathbb R$ is an algebraic number of degree $N$, then clearly all $\alpha_n$ are real in the field $K=\mathbb Q(\alpha)$. We will denote $\sigma_1(\beta),\dots,\sigma_{N-1}(\beta)$ the conjugates of $\beta\in K$ (different from it), and write
\[
f_\beta(x)=C_\beta(x-\beta)(x-\sigma_1(\beta))\dotsb(x-\sigma_{N-1}(\beta))\in\mathbb Z[x]
\]
for its minimal polynomial. We will call an algebraic number $\beta\in K$ a reduced algebraic irrationality if $\beta$~is real, $\beta>1$ and its conjugates $\sigma_j(\beta)$ for $j=1,\dots,N-1$ are all located in the disk $|z+\frac12|<\frac12$ (whose boundary is the circle placed on the diameter $[-1,0]$).
Notice that such a reduced irrationality $\beta$ must be an algebraic number of degree exactly $N$ over~$\mathbb Q$, as $\beta$ differs from all its conjugates.
In the case when $K$ is a totally real field, the condition on the conjugates means that $-1<\sigma_j(\beta)<0$ for $j=1,\dots,N-1$.

The following statement generalises \cite[Theorem~4.8]{Zu23} to the case of algebraic irrational $\alpha\in\mathbb R$.

\begin{theorem}
\label{th1}
Let $\alpha$ be a real algebraic number of degree $N>1$.
Then
\begin{enumerate}[(iii)]
\item[\textup{(i)}]
the number $\alpha_n$, where $n\ge1$, in the continued fraction
\[
\alpha=[a_1,\dots,a_{n-1},\alpha_n]
\]
has the same discriminant as~$\alpha$\textup;
\item[\textup{(ii)}]
if $\alpha$ is a reduced irrationality then $\alpha_n$ is reduced for any $n\ge1$ as well\textup;
and
\item[\textup{(iii)}]
if $\alpha$ is not necessarily reduced then $\alpha_n$ is reduced for all $n$ sufficiently large.
\end{enumerate}
\end{theorem}

\begin{proof}
(i) This follows from the fact that $GL_2(\mathbb Z)$-transformations
\[
\gamma=\begin{pmatrix} a & b \\ c & d \end{pmatrix}\colon\mathbb Z[x]\to\mathbb Z[x],
\quad
f(x)\mapsto(cx+d)^{\deg f(x)}f\bigg(\frac{ax+b}{cx+d}\bigg),
\]
preserve the discriminant.

(ii) Write $\alpha=a+1/\beta$ with the integer $a=\lf\alpha\rf$ and real $\beta>1$, so that the constraint $\beta>1$ is automatically met (even if $\alpha$~is not reduced).
If $\alpha$ is reduced then $a\ge1$; now $\beta=-1/(a-\alpha)$, hence
\[
\sigma_j(\beta)=-\frac1{a-\sigma_j(\alpha)} \quad\text{for}\; j=1,\dots,N-1.
\]
It remains to notice that the image of the disk $|z+\frac12|<\frac12$ under the (linear-fractional) map $z\mapsto-1/(a-z)$ is the disk whose boundary is the circle placed on the diameter $[-1/a,-1/(a+1)]\subset[-1,0)$, that is, the disk
\[
\bigg|z+\frac{2a+1}{2a(a+1)}\bigg|<\frac1{a(a+1)}.
\]
Such disks are part of $|z+\frac12|<\frac12$ for every $a=1,2,\dots$\,.

(iii) From \eqref{eq-lf2} we obtain
\[
\sigma_j(\alpha_{n+1})=-\frac{q_{n-1}\sigma_j(\alpha)-p_{n-1}}{q_n\sigma_j(\alpha)-p_n}
=-\frac{q_{n-1}}{q_n}\,\frac{\sigma_j(\alpha)-p_{n-1}/q_{n-1}}{\sigma_j(\alpha)-p_n/q_n}
\]
for each $j=1,\dots,N-1$.
Furthermore, from the standard properties of convergents we have
\[
\frac{\sigma_j(\alpha)-p_{n-1}/q_{n-1}}{\sigma_j(\alpha)-p_n/q_n}
=1+\frac{p_n/q_n-p_{n-1}/q_{n-1}}{\sigma_j(\alpha)-p_n/q_n}
=1+\frac{(-1)^n}{q_nq_{n-1}(\sigma_j(\alpha)-p_n/q_n)}
\]
implying that
\[
\sigma_j(\alpha_{n+1})+\frac12
=\frac12\bigg(1-2\frac{q_{n-1}}{q_n}+\frac{2\cdot(-1)^n}{q_nq_{n-1}(\sigma_j(\alpha)-p_n/q_n)}\bigg).
\]
Since $p_n/q_n\to\alpha$ as $n\to\infty$, for all $n$ sufficiently large we obtain
\[
\bigg|\sigma_j(\alpha_{n+1})+\frac12\bigg|
\le\frac12\bigg(\bigg|1-2\frac{q_{n-1}}{q_n}\bigg|+\frac{C}{q_nq_{n-1}}\bigg)
\]
with some absolute constant $C>0$.
If now $a_n\ge2$ then $q_n=a_nq_{n-1}+q_{n-2}>2q_{n-1}$ implying
\[
\bigg|1-2\frac{q_{n-1}}{q_n}\bigg|=1-2\frac{q_{n-1}}{q_n};
\]
otherwise $a_n=1$ results in $q_{n-1}=q_n-q_{n-2}$, hence
\[
\bigg|1-2\frac{q_{n-1}}{q_n}\bigg|=\bigg|1-2\frac{q_n-q_{n-2}}{q_n}\bigg|
=\bigg|-1+2\frac{q_{n-2}}{q_n}\bigg|=1-2\frac{q_{n-2}}{q_n}.
\]
In both situations we deduce that
\[
\bigg|1-2\frac{q_{n-1}}{q_n}\bigg|+\frac{C}{q_nq_{n-1}}
=1-2\frac{q_{n-k}}{q_n}+\frac{C}{q_nq_{n-1}}<1
\]
for $n$ sufficiently large, independent of whether $k=1$ or~$2$.
Thus, for all such $n$ we get
\[
\bigg|\sigma_j(\alpha_{n+1})+\frac12\bigg|<\frac12, \quad\text{where}\; j=1,\dots,N-1,
\]
completing the proof of the theorem.
\end{proof}

\begin{remark}
One can simplify the final part of the argument in the proof of claim~(iii) by observing that $a_n\ge2$ for some $n=n_0$.
This would imply that $\alpha_{n_0+1}$ is a reduced algebraic irrationality but then $\alpha_{n+1}$ is reduced for all $n\ge n_0$ by part~(ii) of the theorem.
\end{remark}

\begin{remark}
If $\alpha$ in Theorem~\ref{th2} has degree $m\ge2$ and $f_0(x)$ denotes its minimal polynomial, then the minimal polynomial $f_n(x)$ of $\alpha_{n+1}$ is up to sign
\[
(q_nx+q_{n-1})^mf_0\bigg(\frac{p_nx+p_{n-1}}{q_nx+q_{n-1}}\bigg)
\]
for $n=0,1,2,\dots$, so that the leading coefficient of it up to sign happens to be $q_n^mf_0(p_n/q_n)$.
When the latter is equal to $\pm1$ (which is an extremely rare event!) and $\alpha_{n+1}$ is reduced, then $\alpha_{n+1}$~is a Pisot number.
\end{remark}

\begin{theorem}
\label{th2}
Let $\alpha\in\mathbb R$ be a cubic irrationality and $\alpha_{n+1}$ the tails in its continued fraction $\alpha=[a_1,\dots,a_n,\alpha_{n+1}]$, where $n=0,1,2,\dots$\,.
Then for the conjugates $\sigma_1(\alpha_{n+1}),\sigma_2(\alpha_{n+1})$ we have the following limit relation\textup:
\begin{equation}
q_n^2\cdot|\sigma_1(\alpha_{n+1})-\sigma_2(\alpha_{n+1})|\to\beta=\bigg|\frac{\sigma_1(\alpha)-\sigma_2(\alpha)}{(\alpha-\sigma_1(\alpha))(\alpha-\sigma_2(\alpha))}\bigg| \quad\text{as}\; n\to\infty.
\label{eq-lim}
\end{equation}
\end{theorem}

\begin{proof}
We use \eqref{eq-lf2} to write
\begin{align*}
\sigma_2(\alpha_{n+1})-\sigma_1(\alpha_{n+1})
&=\frac{q_{n-1}\sigma_1(\alpha)-p_{n-1}}{q_n\sigma_1(\alpha)-p_n}-\frac{q_{n-1}\sigma_2(\alpha)-p_{n-1}}{q_n\sigma_2(\alpha)-p_n}
\\
&=\frac{(q_{n-1}\sigma_1(\alpha)-p_{n-1})(q_n\sigma_2(\alpha)-p_n)-(q_n\sigma_1(\alpha)-p_n)(q_{n-1}\sigma_2(\alpha)-p_{n-1})}{(q_n\sigma_1(\alpha)-p_n)(q_n\sigma_2(\alpha)-p_n)}
\\
&=\frac{(p_nq_{n-1}-p_{n-1}q_n)(\sigma_2(\alpha)-\sigma_1(\alpha))}{(q_n\sigma_1(\alpha)-p_n)(q_n\sigma_2(\alpha)-p_n)}
\\
&=\frac{(-1)^n(\sigma_2(\alpha)-\sigma_1(\alpha))}{q_n^2(\sigma_1(\alpha)-p_n/q_n)(\sigma_2(\alpha)-p_n/q_n)}.
\end{align*}
Recalling that $p_n/q_n\to\alpha$ as $n\to\infty$ leads to the required claim.
\end{proof}

\begin{remark}
The fact that $|\sigma_1(\alpha_{n+1})-\sigma_2(\alpha_{n+1})|<C/q_n^2$ for all $n$ is shown by A.~Sch\"onhage in the proof of Theorem~2 in~\cite{Sc06}.
\end{remark}

\begin{remark}
Theorem~\ref{th2} can be compared with the following result for real quadratic irrationalities (whose continued fractions are, of course, eventually periodic): the continued fraction expansions of a real quadratic irrationality and its Galois conjugate have essentially the same pre-periodic part.
The details can be found in \cite[Section~6]{Bu12}.
\end{remark}

In the setup of Theorem~\ref{th2}, let $f_n(x)$ denote the minimal polynomial of $\alpha_{n+1}$ for $n=0,1,2,\dots$\,.
By \eqref{eq-lf1} we have
\[
f_n(x)=(q_nx+q_{n-1})^3f_0\bigg(\frac{p_nx+p_{n-1}}{q_nx+q_{n-1}}\bigg),
\]
so that its leading coefficient $C_n$ is precisely $(-1)^nq_n^3f_0(p_n/q_n)$.
Using
\[
C_n^2|(\alpha_{n+1}-\sigma_1(\alpha_{n+1}))(\alpha_{n+1}-\sigma_2(\alpha_{n+1}))(\sigma_1(\alpha_{n+1})-\sigma_2(\alpha_{n+1}))|=|D(\alpha_{n+1})|^{1/2}=|D(\alpha)|^{1/2},
\]
where $D(\,\cdot\,)$ denotes the discriminant of algebraic number (that is, of its minimal polynomial in $\mathbb Z[x]$), we deduce from Theorem~\ref{th2} that
\[
|(\alpha_{n+1}-\sigma_1(\alpha_{n+1}))(\alpha_{n+1}-\sigma_2(\alpha_{n+1}))|\sim\frac{|D(\alpha)|^{1/2}}{\beta}\bigg(\frac{q_n}{C_n}\bigg)^2
\quad\text{as}\; n\to\infty,
\]
and even that
\begin{equation}
|\alpha_{n+1}-\sigma_j(\alpha_{n+1})|\sim
\frac{|D(\alpha)|^{1/4}}{\beta^{1/2}}\,\frac{q_n}{|C_n|}
=\frac{|D(\alpha)|^{1/4}}{\beta^{1/2}}\,\frac1{q_n^2|f_0(p_n/q_n)|}
\quad\text{as}\; n\to\infty,
\label{eq-asym}
\end{equation}
for $j=1,2$.
In a particular example of $\alpha\in\{2\cos\frac{2\pi}7,2\cos\frac{4\pi}7,2\cos\frac{6\pi}7\}$ being a root of the polynomial $x^3 + x^2 - 2x - 1$ of discriminant~49 (in this case $K=\mathbb Q(x_1)=\mathbb Q(x_2)=\mathbb Q(x_3)$ is cyclic totally real), the quantity $|D(\alpha)|^{1/4}/\beta^{1/2}=\sqrt{7/\beta}$ is, up to sign, a root of the cubic polynomial $x^3 - 7x^2 + 49$.
Performing a similar computation for $\alpha=\sqrt[3]{2}$ (whose discriminant is $-108$) brings in $\beta=1/\sqrt[6]{108}$ and $|D(\alpha)|^{1/4}/\beta^{1/2}=\sqrt[3]{108}$.

\medskip
The following fact is mentioned, without proof or reference, in \cite[Section~8]{Ch83}.
We include its derivation as well.

\begin{lemma}
\label{lemX}
Let $K=\mathbb Q(\beta)$ be a cubic field. Then every $\alpha\in K$ can written as
\[
\alpha=\frac{a\beta+b}{c\beta+d} \quad\text{with some}\; a,b,c,d\in\mathbb Z;
\]
furthermore, $\alpha\notin\mathbb Q$ if and only if $ad-bc\ne0$ in this representation.
\end{lemma}

\begin{proof}
Write $\alpha=A_0+A_1\beta+A_2\beta^2$ and $\beta^3=B_0+B_1\beta+B_2\beta^2$, where $A_0,\dots,B_2\in\mathbb Q$.
The question is to determine $a,b,c,d\in\mathbb Z$ such that $(a\beta+b)/(c\beta+d)=\alpha$, equivalently,
\begin{align*}
a\beta+b
&=(A_0+A_1\beta+A_2\beta^2)(c\beta+d)
\\
&=A_0d+(A_0c+A_1d)\beta+(A_1c+A_2d)\beta^2+A_2c\beta^3
\\
&=(A_0d+A_2B_0c)+(A_0c+A_1d+A_2B_1c)\beta+(A_1c+A_2d+A_2B_2c)\beta^2.
\end{align*}
This leads to the system of linear homogeneous equations
\[
\begin{pmatrix} 0 & -1 & A_2B_0 & A_0 \\ -1 & 0 & A_0+A_2B_1 & A_1 \\ 0 & 0 & A_1+A_2B_2 & A_2 \end{pmatrix}
\begin{pmatrix} a \\ b \\ c \\ d \end{pmatrix}
=\begin{pmatrix} 0 \\ 0 \\ 0 \end{pmatrix}
\]
with rational coefficients. Thus, a required solution in $a,b,c,d\in\mathbb Z$ exists when the rank of the $3\times4$ matrix is equal to~3, that is, when either $A_2$ or $A_1$ does not vanish.
When $A_1=A_2=0$, in other words, when $\alpha=A_0\in\mathbb Q$, the choice $a=c=0$ and $b/d=A_0$ makes the job.

The non-vanishing of $ad-bc$ in the case of $\alpha$ irrational is a simple observation.
\end{proof}

Now recall that a real number $\alpha$ is called badly approximable if $|\alpha-p/q|>C/q^2$ for some $C=C(\alpha)>0$ and all rational $p/q$; in other words, if
\begin{equation}
\Lambda=\Lambda(\alpha)=\liminf_{q\to\infty}\min_{p\in\mathbb Z}q|q\alpha-p|>0.
\label{eq-lam}
\end{equation}
If this is not the case, $\alpha$ is called well approximable.
The badly approximable numbers are precisely those with bounded partial quotients in their continued fraction.

The following fact is straightforward.

\begin{lemma}
\label{lemY}
Let $\alpha,\beta$ be real and
\[
\alpha=\frac{a\beta+b}{c\beta+d} \quad\text{with some}\; a,b,c,d\in\mathbb Z, \; ad-bc\ne0.
\]
Then $\alpha$ is badly approximable if and only if $\beta$~is.
\end{lemma}

\begin{proof}
This follows from the fact that best rational approximations $p/q$ to $\beta$ with $|q|$ sufficiently large correspond to best approximations $(ap+bq)/(cp+dq)$ to~$\alpha$.
Furthermore, the condition $ad-bc\ne0$ implies that the roles of $\alpha$ and $\beta$ can be swapped.
\end{proof}

\begin{remark}
It is also possible to quantify the result of Lemma~\ref{lemY} for badly approximable $\alpha$ (and~$\beta$) in terms of the characteristic~\eqref{eq-lam}.
If $\{p/q\}=\{p_n/q_n\}$ is a sequence of rational approximations to $\alpha$ that realises the limit inferior in~\eqref{eq-lam}, that is, for which $|\alpha-p/q|\sim\Lambda(\alpha)/q^2$ as $q\to\infty$, then
\[
\bigg|q\,\frac{a\beta+b}{c\beta+d}-p\bigg|=|q\alpha-p|\sim\frac{\Lambda(\alpha)}{q},
\]
hence
\[
|q(a\beta+b)-p(c\beta+d)|=|(aq-cp)\beta-(dp-bq)|\sim\Lambda(\alpha)\bigg|\frac{c\beta+d}{q}\bigg|.
\]
On the other hand,
\[
\beta=\frac{d\alpha-b}{-c\alpha+a}\sim\frac{dp-bq}{-cp+aq},
\]
so that
\[
\bigg|\frac{c\beta+d}{q}\bigg|\sim\frac{|ad-bc|}{|aq-cp|}
\]
and we end up with
\[
\bigg|\beta-\frac{dp-bq}{aq-cp}\bigg|\sim\Lambda(\alpha)\frac{|ad-bc|}{(aq-cp)^2}.
\]
This inequality implies that $\Lambda(\beta)\le|ad-bc|\Lambda(\alpha)$; from the symmetry of $\alpha$ and $\beta$ we also get $\Lambda(\alpha)\le|ad-bc|\Lambda(\beta)$.
Not surprisingly the condition $|ad-bc|=1$ results in the equality $\Lambda(\beta)=\Lambda(\alpha)$; the numbers $\alpha$ and $\beta$ are $GL_2(\mathbb Z)$-equivalent in this case, hence they share the same continued-fraction tails by the classical theorem of Serret \cite[Theorem~4.6]{Zu23}.
\end{remark}

As a consequence of Lemmas~\ref{lemX} and \ref{lemY} we obtain the following observation.

\begin{theorem}
\label{th3}
Let $\alpha$ be a real cubic irrationality and $K=\mathbb Q(\alpha)$.
Then an irrational $\beta\in K$~is badly approximable if and only if $\alpha$~is.

In other words, if the partial quotients of \emph{some} real number from a cubic field are unbounded, then this is true for \emph{any} real irrational number from the field.
\end{theorem}

By the asymptotics in \eqref{eq-asym} the unboundedness of partial quotients follows from
\[
\liminf_{q\to\infty}\min_{p\in\mathbb Z}q^2f_0(p/q)=0,
\]
where $f_0(x)\in\mathbb Z[x]$ is the minimal polynomial of~$\alpha$.
However the latter criterion of well-approximability is classically known from the work of Thue and Siegel.

Finally notice that Theorems~\ref{th2} and \ref{th3} do not extend to (generic) algebraic irrationalities of degree higher than~3.

\medskip\noindent
\textbf{Acknowledgements.}
I thank Yann Bugeaud, Alan Haynes and Ariyan Javanpeykar for inspiring discussions, corrections and numerous comments incorporated as remarks in the present text.

\end{document}